\documentclass[sn-mathphys,Numbered]{sn-jnl}

\usepackage{graphicx}%
\usepackage{multirow}%
\usepackage{amsmath,amssymb,amsfonts}%
\usepackage{amsthm}%
\usepackage{mathrsfs}%
\usepackage[title]{appendix}%
\usepackage{xcolor}%
\usepackage{textcomp}%
\usepackage{manyfoot}%
\usepackage{booktabs}%
\usepackage{algorithm}%
\usepackage{algorithmicx}%
\usepackage{algpseudocode}%
\usepackage{listings}%

\usepackage[utf8]{inputenc}
\usepackage[official]{eurosym}
\usepackage{mathtools}
\usepackage{dirtytalk}
\usepackage{tikz}
\usetikzlibrary{backgrounds}
\usetikzlibrary{shapes.geometric, arrows,calc,math,decorations.pathreplacing}
\usepackage[disable]{todonotes}
\usetikzlibrary{patterns}
\usetikzlibrary{shapes.misc}
\usepackage{xspace}
\usepackage[shortlabels]{enumitem}
\usepackage[short]{optidef}
\usepackage[capitalise]{cleveref}
\usepackage{url}
\usepackage{caption}
\usepackage{subcaption}

\definecolor{rwthorange}{RGB}{246,168,0}
\definecolor{rwthmaigruen}{RGB}{189,205,0}
\definecolor{rwthtuerkis}{RGB}{0,152,161}
\newcommand{\fcolor}{rwthorange}
\newcommand{\mcolor}{rwthtuerkis}
\tikzset{square/.style={regular polygon, regular polygon sides=4}}
\tikzset{patient/.style={square,fill=#1,draw=#1, minimum width=0.95cm}}
\tikzset{stay/.style={#1,line width=4.5pt}}
\tikzset{room/.style={square,fill=none,draw=black,minimum width=0.90cm}}
\tikzset{note/.style={draw=none,fill=none}}

\newcommand{\patient}[4]{
    \foreach \i [evaluate=\i as \m using \i-1] in {1,...,#2}{
        \node[patient=#1](#3\i) at ($#4+(\m*1,0)$) {}; 
    }   
    \ifnum#2>1
        \foreach \i [evaluate=\j as\m using \i-1] in {2,...,#2}{
            \draw[stay=#1] (#3\m) -- (#3\i);
        }   
    \fi 
}
\newcommand{\fpatient}[4][]{\patient{\fcolor #1}{#2}{#3}{#4}}
\newcommand{\mpatient}[4][]{\patient{\mcolor #1}{#2}{#3}{#4}}
\newcommand{\tikztransfer}[3]{
        \draw[stay=#1] (#2.center) -- (#3.center);
}
\newcommand{\ftransfer}[2]{\tikztransfer{\fcolor}{#1}{#2}}

\newcommand{\printRooms}[2]{    
    \tikzmath{\x=0;}
    \tikzmath{\i=0;}
    \foreach \nb/\cap in {#1} {
        \foreach \r in {1,...,\nb}{
            \pgfmathparse{int(\i+1)}
            \xdef\i{\pgfmathresult}
            \draw[very thick] (-0,-\x) -- (#2,-\x);
            \foreach \b in {1,...,\cap}{
                \pgfmathparse{\x+1}
                \xdef\x{\pgfmathresult}
                \foreach \e [evaluate=\e as \d using \e-1] in {1,...,#2}{
                    \node[room](b\i/\b/\e) at (\d+0.5,-\x+0.5) {};
                }
            }
        }
    }
    \draw[very thick] (-0,-\x) -- (#2,-\x);
}
\newcommand{\labelNode}[2]{\node[fill=none,draw=none] at (#1) {\bfseries\large $#2$};}

\setlist{topsep=5pt,leftmargin=33pt}

\newcommand{\gitlink}{\url{https://github.com/TabeaBrandt/patient-to-room_assignment/}}

\definecolor{RWTHblue}{RGB}{0,83,159}
\definecolor{RWTHblack}{RGB}{0,0,0}
\definecolor{RWTHwhite}{RGB}{255,255,255}
\definecolor{RWTHlightblue}{RGB}{142,186,226}
\definecolor{RWTHgrey}{RGB}{51,51,51}
\definecolor{RWTHlightgrey}{RGB}{204,204,204}
\definecolor{RWTHsuperlightgrey}{RGB}{247,247,247}
\definecolor{RWTHpetrol}{RGB}{0,97,101}
\definecolor{RWTHteal}{RGB}{0,152,161}
\definecolor{RWTHmaygreen}{RGB}{189,205,0}
\definecolor{RWTHgreen}{RGB}{87,171,39}
\definecolor{RWTHyellow}{RGB}{255,237,0}
\definecolor{RWTHorange}{RGB}{246,168,0}
\definecolor{RWTHmagenta}{RGB}{227,0,102}
\definecolor{RWTHred}{RGB}{204,7,30}
\definecolor{RWTHbordeaux}{RGB}{161,16,53}
\definecolor{RWTHviolet}{RGB}{97,33,88}
\definecolor{RWTHpurple}{RGB}{122,111,172}

\newtheoremstyle{mystyle}
{3pt}
{3pt}
{\itshape}
{}
{\bfseries}
{}
{.5em}
{}
\theoremstyle{mystyle}
\newtheorem{lemma}{Lemma}
\newtheorem{definition}{Definition}

\makeatletter
\newcommand{\setword}[2]{%
  \phantomsection
  #1\def\@currentlabel{\unexpanded{#1}}\label{#2}%
}
\makeatother

\renewcommand{\P}{\mathcal{P}}
\newcommand{\priv}{\P^*}         
\newcommand{\R}{\mathcal{R}}
\newcommand{\nR}{R}
\newcommand{\T}{\mathcal{T}}
\newcommand{\nT}{T}
\newcommand{\fp}{\P^{\mathrm{f}}}  
\renewcommand{\mp}{\P^{\mathrm{m}}}  
\newcommand{\F}{\mathcal{F}}          
\newcommand{\rpold}{\F}          
\newcommand{\nfp}{F_t}
\newcommand{\nmp}{M_t}
\newcommand{\nprivf}{\nfp^*}
\newcommand{\nprivm}{\nmp^*}

\newcommand{\arr}{a}    
\newcommand{\dis}{d}    

\newcommand{\transobj}{f^\mathrm{trans}}    
\newcommand{\privobj}{f^\mathrm{priv}}      
\newcommand{\ftrans}{f^\mathrm{trans}}    
\newcommand{\fpriv}{f^\mathrm{priv}}      

\newcommand{\rc}{c}     
\newcommand{\ass}{z}    

\newcommand{\smax}{s^{\max}}

\DeclarePairedDelimiter\ceil{\lceil}{\rceil}

\newcommand{\N}{\mathbb{N}}
\newcommand{\abs}[1]{|#1|}
\newcommand{\pluseq}{\mathrel{{+}{=}}}
\newcommand{\np}{\mathcal{NP}}
\newcommand{\ppp}{PPP\xspace}

\begin{document}

\title[Structural insights about assigning single rooms in the patient-to-room assignment problem]{Structural insights about assigning single rooms in the patient-to-room assignment problem and an IP-based solution method}


\author*[1]{\fnm{Tabea} \sur{Brandt}}\email{brandt@combi.rwth-aachen.de}

\author[1]{\fnm{Christina} \sur{Büsing}}\email{buesing@combi.rwth-aachen.de}

\author[1]{\fnm{Felix} \sur{Engelhardt}}\email{engelhardt@combi.rwth-aachen.de}
\equalcont{These authors contributed equally to this work.}

\affil*[1]{\orgdiv{Department of Computer Science}, \orgname{RWTH Aachen}, \orgaddress{\street{Ahornstraße 55}, \city{Aachen}, \postcode{52074}, \state{Aachen}, \country{Germany}}}


\abstract{Patient-to-room assignment (PRA) is a scheduling problem in decision support for hospitals.
It consists of assigning patients to rooms according to certain objectives, e.g., avoiding transfers and respecting single-room requests.
This work presents combinatorial insights about the feasibility of PRA and about the assignment of patients to single rooms.
We further compare different IP-formulations for PRA as well as the influence of different objectivs on the runtime.
Based on these results, we develop a fast IP-based solution approach which obtains high quality solution.
The applicability is verified through a computational study with instances derived from real-world data.
Results indicate that large, real world instances can be solved to a high degree of optimality within (fractions of) seconds.
}


\maketitle

\clearpage
\section{Introduction and problem definition}
\label{problemdefinition}
Beds and rooms for patients are important resources in hospitals and the decision which bed and room a patient occupies
impacts not only the staff's workload \cite{blay2017_bedtransfers}, but also patient satisfaction \cite{transferfun}, and the provision of surcharges \cite{transferprice}. 
The assignment of patients to beds and rooms is usually either performed by so-called case managers or by experienced nurses.
In literature, both the terms patient-to-room assignment problem (PRA) and patient-to-bed assignment problem (PBA) have been used to describe this task.
Here, the term bed is used synonymously for bed space.
In general, there are different bed types, e.g., for small children or heavy weight patients which are provided as rolling stock.
A room's bed spaces, however, can be considered as equal.
The task of finding a physical bed of appropriate bed type for a patient is independent of assigning the patient to a room/bed space.
We therefore use the term PRA to avoid confusion.

Typically, there are two types of case management systems in hospitals: centralised and decentralised systems.
In a centralised system, all patient-to-room assignments are decided by the same person or work group.
Whereas in a decentralised system, the patient-to-room assignments are decided on ward or speciality level~\cite{Schmidt2013}.
In both cases, PRA is based on a previously fixed admission scheduling decision.
In the first formal definition of PRA proposed by Demeester et al. in 2010~\cite{DEMEESTER2010}, they considered a centralised system with multiple wards and specialities.
Patients then need a room in a ward with appropriate specialty.
This definition is still often used in literature.
However, we found that in our local hospital a decentralised system is used.
In this work, we present combinatorial insights and a solution approach for a decentralised system where patient-to-room assignments are decided on ward level.
Combinatorially, the decentralised system is a special case of the centralised system.

Another characteristic of the definition proposed by Demeester et al. is that some patients may only be assigned to specific rooms to account for, e.g., equipment requirements~\cite{DEMEESTER2010}.
However, we found that in our local hospital all rooms have the same default equipment and the special equipment is rolling stock.
Therefore, we assume that a ward's rooms are all equal and that every patient can be assigned to every room.
We experience this to be a common setting in German hospitals.

Formally, we consider a ward with rooms $\R$ and $\rc_r\in\N$ beds in room $r\in \R$, as well as a discrete planning horizon $\T=\{1,\ldots,\nT\}$.
In our computational study, we use 24h as length of one time period so that $\nT$ refers to the number of days in the planning horizon.
However, all concepts in this paper are easily transferable to half-day or even smaller planning intervals.
Further, let $\P$ denote the set of all patients.
For every patient $p\in \P$, we know their registration period, their arrival period $\arr_p\in \T_0 := \T \cup \{0\}$, their discharge period $\dis_p\in \T$ satisfying $\arr_p<\dis_p$, their sex, and whether they are entitled to a single room.
Commonly, patients whose arrival and registration periods are identical are called \emph{emergency patients} and otherwise \emph{elective patients}~\cite{bedmanagement_review}.
Here, patients with $\arr_p=0$ are patients who arrived in an earlier time period.
Therefore, those patients already have pre-assigned rooms which are given in the set $\F\subset \{p\in \P \mid \arr_p=0\}\times \R$.
We denote the set of female patients with $\fp\subseteq \P$, the set of male patients with $\mp\subseteq \P$, and the set of patients entitled to a single room with $\priv\subseteq\P$.
Note that we assume $\fp\cap\mp=\emptyset$ and $\P=\fp\cup\mp$ based on the data provided by our local hospital.

The main task in PRA is to assign every patient $p\in \P$ to a room $\ass(p,t)\in \R$ for every time period $\arr_p\leq t<\dis_p$ of their stay.
We assume that all patients stay in hospital on consecutive periods from admission to discharge period and that they are discharged at the beginning of a time period.
Thus, patients do not need a room in their discharge period, which is a common assumption in literature, cf.~\cite{Vancroonenburg2016}.
We define the set of all patients who need a room in time period $t\in \T$ as $\P(t)=\{p\in \P\mid \arr_p\leq t<\dis_p\}$. 
Further, we denote for any subset of patients $S\subset\P$ the subset of patients in need for a bed in time period $t\in \T$ by $S(t):=S\cap \P(t)$.

The assignment $\ass$ of patients to rooms has to fulfil two conditions for every room $r\in \R$ and every time period $t\in \T$ in order to be feasible:
\begin{description}
    \item[\setword{(C)}{Word:C}] room capacities $\rc_r$ are respected, i.e., $|\{p\in \P \mid \ass(p,t)=r\}|\leq \rc_r$,
    \item[\setword{(S)}{Word:G}] female and male patients are assigned to separate rooms, i.e.,
\[\{z(p,t)\mid p\in \fp(t)\} \cap \{z(p,t)\mid p\in \mp(t)\}=\emptyset.\]
\end{description}
Theoretically, these constraints may lead to infeasibility, which is unacceptable in practical application.
However, we assume based on practitioners demands that the case manager ensures respect of the ward's capacity under consideration of the sex separation condition.
Therefore, all considered instances in this paper are feasible under both conditions \ref{Word:C} and \ref{Word:G}.
We discuss how feasibility under these constraints can be checked combinatorially in \cref{sec:feas}.

Real-life optimization problems often have to balance the, potentially conflicting, interests of multiple stakeholders.
For PRA, Schäfer et al. identified patients, doctors and nurses as the main stakeholders~\cite{schaefer2019}.
Based on literature and interviews with members of each of the groups, they concluded that patients primarily desire a pleasant stay, i.e., a bed in a suitable room, without unnecessary transfers or waiting in an overflow area, and suitable roommates.
Doctors primarily look for visiting rounds that minimize walking distances.
In comparison to that, nursing staff emphasise the relevance of a balanced workload~\cite{schaefer2019}.
A common approach, also used by Sch\"afer et al., is to combine the objectives of all stakeholders into one objective function as a weighted sum.
However, the appropriate choice of weights is not obvious and depends strongly on the hospital management's values.
On the contrary, we consider only two objective functions and attempt a thorough investigation of their combinatorial structure, their performance in BIPs and their interoperability.
For this, we consider the objectives both separately and in different hierarchical orders that are motivated by the different stakeholders' points of view. 

Our first objective is to avoid that patients have to change rooms during their stay, so-called \emph{patient transfers}.
Patient transfers increase the staff's workload and reduce patient satisfaction while providing no immediate health benefits for patients \cite{storfjell2009}.
A case study by Blay et al. reports that transfers require on average between $11$min (intra-ward transfer) and $25$min (inter-ward transfer) of direct nursing time  \cite{blay2017_bedtransfers}. 
Additionally, there are several ways \cite{fekieta} in which these transfers can put patients health at risk, e.g., by leading to delays in care \cite{johnson}, interruptions in treatment \cite{papson} and increased infections \cite{blay2017_adverseOutcomes}. 
Therefore, we minimise the total number of patient transfers
\[
\ftrans:= \sum_{p \in \P}\left(\sum_{t=\arr_p}^{\dis_p-2} |\{\ass(p,t),\ass(p,t+1)\}|-1\right).
\]

Another possibility of addressing the topic of patient transfers is to minimise the maximum number of transfers per patient.
According to Brandt et al., those two interpretations of avoiding transfers are not conflicting but can be optimised simultanously for the case of double rooms~\cite{1perpatient}.
More precisely, they showed that there always exists an optimal solution with regard to $\ftrans$ where each patient is transferred at most once.
Therefore, we choose the minimisation of the total number of transfers as objective function.
Our computational experiments with real life data showed that in an optimal solution with respect to $\ftrans$ no patient is transferred twice regardless of an upper bound on the number of transfers per patient.
Moreover, the enforcement of an upper bound on the number of transfers per patient did not affect the runtime.
Therefore, we exclusively consider $\transobj$ as objective function for transfers.

Our second objective is the assignment of single rooms to patients who need isolation for medical reasons or who are entitled to one because of a private health insurance (\emph{private patients}).
In practice, medical reasons have priority but also the latter case is of high interest for the hospital management as such additional services provide income opportunities, with, e.g., a single room surcharge numbering $175$\EUR{} per day \cite{ukafees}.
Unfortunatly however, we lack data on which patients require isolation for medical reasons.
Therefore, we focus here on the assignment of single rooms to private patients.
In Germany, the fees for a single room are paid by the insurance companies for every day individually.
Therefore, we maximise the total number of time periods that private patients spend alone in a room, i.e.,
\[
\fpriv:= \sum_{t \in \T}\left(\sum_{\substack{p \in \priv(t)}} 1-\min\left\{1,~\abs{\{q\in \P(t)\setminus \{p\} \mid \ass(p,t)=\ass(q,t)\}}\right\}\right).
\]
Remark that all our results can easily be extended to incorporate medically necessary isolation either as an additional objective, analagously to $\fpriv$ but with higher priority, or as an additional feasibility condition.

\todo[inline]{Discuss whether we need \cref{fig:fig_intro} as long as no other figures in this style are added.}
\cref{fig:fig_intro} illustrates the role of both objectives. Note that, in general, transfers are necessary for feasibility in PRA. In the example, two rooms and a time horizon of three time steps are given. In step one, both rooms are assigned two male/female patients each. After the first step, the two male patients leave and a third female patient arrives. Now, as two more male patients arrive in the third step, a transfer is necessary to ensure feasibility. 
The example also includes a private patient which is marked by a $*$. Here, the private patient can be charged a single room surcharge only for the second time step.
Hence, for the assignment depicted in the example we have $\fpriv=\ftrans=1$ and this is optimal for both objectives.

\begin{figure}[htb] \centering
\resizebox{!}{0.2\textheight}{
    \begin{tikzpicture}

\begin{scope}
\begin{pgfonlayer}{background}
    \printRooms{2/2}{3}
\end{pgfonlayer}
\mpatient{1}{a}{(b2/1/1)}
\mpatient{1}{b}{(b2/2/1)}
\fpatient{1}{c}{(b1/2/1)}
\fpatient{3}{d}{(b1/1/1)}
\fpatient{2}{e}{(b1/2/2)}
\fpatient{1}{c'}{(b2/1/2)}
\ftransfer{c1}{c'1}
\labelNode{c1}{*}
\mpatient{1}{f}{(b2/1/3)}
\mpatient{1}{g}{(b2/2/3)}


\node[draw=none] at (1.5,-4.5) {assignment};

\draw[decoration={brace,amplitude=8pt,raise=2pt},decorate,very thick](3,0) -- node[right=10pt] { room $r_1$} (3,-2); 
\draw[decoration={brace,amplitude=8pt,raise=2pt},decorate,very thick](3,-2) -- node[right=10pt] { room $r_2$} (3,-4); 

\node at(-1,-0.5) {bed $b_0$};
\node at(-1,-1.5) {bed $b_1$};
\node at(-1,-2.5) {bed $b_2$};
\node at(-1,-3.5) {bed $b_3$};
\end{scope}

\begin{scope}[shift={(-5.5,0)}]
\draw[very thick] (0,0) -- (3,0);
\foreach \i in {0,...,3}{
    \node[rectangle,minimum width=0cm,minimum height=0.4cm,draw,inner sep=0pt] at (\i,0) {}; }
\node[draw=none,label=above:{1}] at (0.5,0) {}; 
\node[draw=none,label=above:{2}] at (1.5,0) {}; 
\node[draw=none,label=above:{$\T=3$}] at (2.5,0) {}; 
\fpatient{3}{d}{(0.5,-0.5)}
\fpatient{2}{c}{(0.5,-1.5)}
\mpatient{1}{f}{(2.5,-1.5)}
\mpatient{1}{a}{(0.5,-2.5)}
\fpatient{2}{e}{(1.5,-2.5)}
\labelNode{c1}{*}
\mpatient{1}{b}{(0.5,-3.5)}
\mpatient{1}{g}{(2.5,-3.5)}

\node[draw=none] at (1.5,-4.5) {instance};
\end{scope}

\begin{scope}[shift={(6.0,0)}]
\fpatient{1}{f}{(0,-0.5)}
\node[draw=none, label={[label distance=0.5cm]0:{female patient}}] at (f1) {}; 
\mpatient{1}{m}{(0,-1.5)}
\node[draw=none, label={[label distance=0.5cm]0:{male patient}}] at (m1) {};
\patient{white}{1}{g}{(0,-2.5)}
\labelNode{g1}{*}
\node[draw=none, label={[label distance=0.5cm]0:{private patient}}] at (g1) {};
\node[room](r) at (0,-3.5) {};
\node[draw=none, label={[label distance=0.5cm]0:{empty bed}}] at (r) {};
\fpatient[,minimum width=0.7cm]{1}{g}{(0.3,-4.5)}
\fpatient[,minimum width=0.7cm]{1}{g'}{(-0.3,-5.2)}
\ftransfer{g1}{g'1}
\node[draw=none, label={[label distance=0.5cm]0:{patient transfer}}] at (0,-5.00) {};
\end{scope}

\end{tikzpicture}
    }
    \caption{Example for $\T=3$ with one private patient where a patient transfer is necessary for
    feasibility}
    \label{fig:fig_intro}
\end{figure}
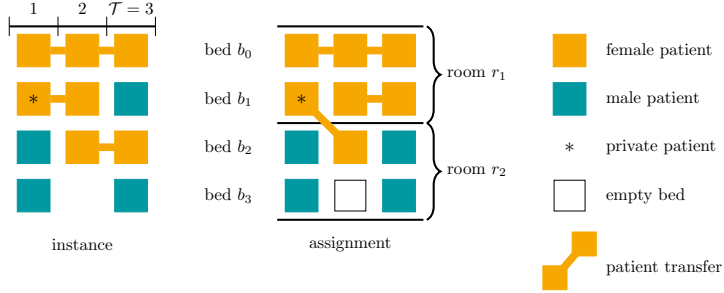

The remainder of this paper is organized as follows.
In \cref{sec:literaturereview}, we give an overview of existing research on integer programming in the context of PRA.
In \cref{sec:combi}, we present combinatorial insights into both feasibility and the maximum number of private patients that can be assigned a single bed each day. 
Then, in \cref{sec:ip:general} we propose and compare multiple IP formulations for PRA.
The computational evaluation shows that in most our (real-life) instances, no transfers are necessary.
Building on that, we propose and compare a second set of IP formulations that contain no transfers in \cref{sec:ip:notransfer}.
In \cref{sec:modelcombi}, we combine the best performing IP formulations with our combinatorial insights from \cref{sec:combi} to solve a dynamic version of PRA with a rolling-time-horizon approach.
Although PRA  is known to be  $\mathcal{NP}$-hard \cite{1perpatient}, we find solutions that are optimal or close to optimality for instances derived from real-world data.
Furthermore, on average, or algorithm requires less than a second per day to find high quality solutions for realistically sized instances.
Finally \cref{sec:further_research}, we point out multiple directions for further research.

\section{Literature review}\label{sec:literaturereview}

In 2010, Demeester et al.
provided the first formal definition of a PRA problem and
proposed a tabu-search algorithm for what they called the \say{problem of automatically and dynamically assigning patients to beds in a general hospital setting} \cite{DEMEESTER2010}.
According to their problem definition, patients have to be assigned to suitable rooms respecting numerous equipment, specialism, and age constraints.
However, limiting a patient's room choices immediately renders the task of assigning patients to beds/rooms $\np$-complete~\cite{1perpatient}.
Ceschia and Schaerf extended the definition by Demeester et al. to include dynamic admission, operating room constraints, time horizons, and patient delays \cite{ceschia2009,ceschiaOKT2011,ceschia2016}.

Next, we give an overview of published integer programming approaches in PRA. For a general literature review on PRA we refer to~\cite{brandt2023integrate,summary}.
A frequent pattern in literature on PRA is to use integer programming to formalize the problem defintion, but not to use it integer programming as a solution method.
This may be due to the fact that, in 2010, Demeester et al. considered integer programming as solution approach.
However, the authors dismissed this, as the given formulation did not result in a feasible solution within an hour and even after a week of computation, no optimal solution was obtained using standard solver software \cite{DEMEESTER2010}.
Ceschia and Schaerf also used an exact solver based on integer linear programming as a reference for small instances, while noting its inability to solve larger instances  \cite{Ceschia2012}.

Nonetheless there are several publications that specifically make use of MIP based solution approaches: Schmidt et al. (2013) define a binary integer program based on patients’ LOS and use it to compare an exact approach, using the MIP solver SCIP, with three heuristic strategies \cite{Schmidt2013}.
Range et al. (2014) reformulate Demeester et al.'s patient admission scheduling problem via Dantzig-Wolfe decomposition and apply a heuristic based on column generation to solve it \cite{Range2014}.
Turhan and Bilgen propose two MIP based heuristics which achieve high quality solutions in fast runtimes compared to respective state of the art studies~\cite{turhan2016}. 
Vancroonenburg et al. (2016) extend the patient assignment problem formulation and develop two corresponding online ILP-models.
The first model focuses on newly arrived patients, whereas the second also considers planned future patients.
They then study the effect of uncertainty in the patients’ LOS, as well as the effect of the percentage of emergency patients.
Note that in all of the cases mentioned above, integer programming is used either as a basis for the development of heuristic solutions or as a reference for small instances, but no exact solving of larger real-world instances is attempted.

Some recent publications also employ integer programming to model both PRA and operating-room usage exactly \cite{conforti2018paper1,conforti2018paper2}. 
However, the models include significant simplifications: fixed room-sex assignment, no transfers and a limited time-horizon.

Most recently, Bastos et al. (2019) present an MIP approach to patient admission scheduling problem, which involves assigning patients to beds over a given time horizon so as to maximize treatment efficiency, patient comfort and hospital utilization, while satisfying all necessary medical constraints and taking into consideration patient preferences as much as possible \cite{Bastos2019}.

There also exist solution approaches that are not based on Demeester's problem definition but are inspired by the setting in a specific hospital.
Thomas et al. developed a MIP based deciscion support system that balances 13 objectives~\cite{Bollapragada2013}
Schäfer et al. disallow (non-medically induced) patient transfers but include overflow and patient preferences~\cite{schaefer2019}.
They also model doctor preferences, i.e., homogenous routes, and then solve the model via a greedy look-ahead heuristic.
In a follow-up publication, they focus on emergency patients and integrate them into the model \cite{schaefer2023}.
Brandt et al. propose a MIP based heuristic for integrated planning of patient-to-room and nurse-to-patient assignment~\cite{brandt2023integrate}.
More generally, Rachuba et al. introduce a taxonomy for evaluating integration consisting of three stages: linkage by constraints/restrictions, sequential and completely integrated planning~\cite{rachuba2023integrated}. Here, our work can contribute to multiple levels, with the combinatorial insights facilitating easy linkage by constraints/restrictions and the IP based approach being suitable for fully integrated planning.

Combinatorial insights about patient-to-room assignment and its underlying structure are still rare.
In \cite{1perpatient} the complexity is discussed and insights about the necessity of transfers are provided, including upper bounds on the total number of transfers as well as the maximum number of transfers per patient.

\section{Combinatorial insights}
\label{sec:combi}
In this section, we first present new combinatorial insights regarding the feasibility of instances with single and double rooms which extend the known results on feasibility from Brandt et al.~\cite{1perpatient}.
Second, we present a combinatorial formula to compute the maximum number of private patients who can be feasibly assigned to single rooms.
Both questions can be decided independently for every single time period, since we allow arbitrary many transfers.
Therefore, in this section we consider an arbitrary but fixed time period $t\in\T$ and abbreviate the number of female patients who are in hospital in time period $t$ with $\nfp:=\abs{\fp(t)}$, and respectively the number of male patients, female private patients, and male private patients needing a bed in time period $t$ with $\nmp:=\abs{\mp(t)}$, $\nprivf:=\abs{\fp(t)\cap\priv(t)}$, and $\nprivm:=\abs{\mp(t)\cap\priv(t)}$.
We further denote with
$\nR_\rc:=\abs{\{r\in\R\mid \rc_r=\rc\}}$
the number of rooms with a specific capacity $\rc\in\N$.

\subsection{Feasibility}
\label{sec:feas}
Brandt et al. define the feasibility problem for an arbitrary but fixed time period $t\in\T$ as follows~\cite{1perpatient}.

\begin{definition}[Feasibility Problem]
Given the number of female and male patients $\nfp,\nmp\in\N_0$, and room capacities $\rc_r\in\N$ for $r\in\R$, does there exist a subset $S\subseteq\R$ of rooms such that it can host all female patients while the male patients fit into the remaining rooms, i.e.,
\begin{equation}\label{eq:def:feas}
\sum_{r\in S} \rc_r \geq \nfp \quad \text{and}
\sum_{r\in \R \setminus S} \rc_r \geq \nmp?
\end{equation}
\end{definition}

Brandt et al. prove that the feasibility problem is $\np$-complete in general and solvable in polynomial time for constant room capacities $\rc_r=\rc\in\N$~\cite{1perpatient}.
Clearly, in the common case of rooms with only double rooms an instance is feasible if and only if
\begin{equation}\label{eq:feas:double}
    \ceil*{\frac{\nfp}{2}} + \ceil*{\frac{\nmp}{2}} \le \abs{\R}
\end{equation}
holds true for every time period $t \in \T$~\cite{1perpatient}.
However, this is no longer accurate for wards that have at least one single room in addition to double rooms otherwise.
For those, it suffices to check whether enough beds are available in total.

\begin{lemma}
\label{lem:feas:1-c}
Consider a ward with room capacities $\rc_r\in\{1,c\}$ with $\rc\in\N$ for all rooms $r\in\R$.
Let the number of female and male patients $\nfp,\nmp\in\N_0$,
be given.
If $\nR_1\geq \rc-1$, then the instance is feasible if and only if the number of patients does not exceed the ward's total capacity, i.e., if and only if
\begin{equation}\label{eq:feas:mixed}
    \nfp +\nmp \le \sum_{r \in \R}c_r 
\end{equation}
holds true for every time period $t \in \T$.
\end{lemma}
\begin{proof}
For $\rc=1$, the instance is obviously feasible if and only if \cref{eq:feas:mixed} holds true. Therefore, let $\rc\geq 2$.
If the number of patients exceeds ward's capacity, i.e.,
\cref{eq:feas:mixed} is violated, then the instance is infeasible as at least one patient cannot be assigned to a room without violating the capacity constraint.
Hence, we assume \cref{eq:feas:mixed} to hold true and show that the instance is then feasible by constructing a set $S\subseteq\R$ which satifies~\cref{eq:def:feas}.

Then, we compute the minimum of the number of rooms of capacity $\rc$ we could fill with female patients
\[k:=\min\left\{\left\lfloor\frac{\nfp}{\rc}\right\rfloor,\nR_\rc\right\},\]
and respectively for male patients
\[\ell:=\min\left\{\left\lfloor\frac{\nmp}{\rc}\right\rfloor,\nR_\rc-k\right\}.\]
Remark that $\nR_\rc\geq k+\ell$, $\nfp-\rc k\geq 0$, and $\nmp-\rc \ell\geq 0$ by construction.
On the one hand, if $\nR_1\geq \nfp-\rc k+\nmp-\rc\ell$, i.e., all remaining patients can be assigned to single rooms, then
we define $S:=S'\cup S''$ and
\begin{align*}
 S'&\subseteq\{r\in\R\mid \rc_r=\rc\}\quad\text{with } \abs{S'}=k,\\
 S''&\subseteq\{r\in\R\mid \rc_r=1\}\quad\text{with } \abs{S''}=\nfp-\rc k.
\end{align*}
Then,
\begin{align*}
\sum_{r\in S} \rc_r &= \sum_{r\in S'}\rc_r + \sum_{r\in S''} \rc_r = \rc k + \nfp- \rc k = \nfp,\\
\sum_{r\in \R\setminus S} \rc_r &= \sum_{r\in \R} \rc_r - \sum_{r\in S} \rc_r \stackrel{\cref{eq:feas:mixed}}{\geq} \nfp + \nmp - \nfp = \nmp,
\end{align*}
i.e., the feasibility condition \cref{eq:def:feas} is satisfied and the instance is feasible.

On the other hand, if $\nR_1 < \nfp-\rc k+\nmp-\rc\ell$, then
\begin{align*}
\nR_\rc &\stackrel{\eqref{eq:def:feas}}{\geq} \frac{1}{\rc}(\nfp+\nmp-\nR_1)\\
&>\frac{1}{\rc}(\nfp+\nmp-\nfp+\rc k -\nmp + \rc\ell = k + \ell.
\end{align*}
Therefore, we also have $\nfp - \rc k < \rc$ and $\nmp - \rc \ell < \rc$ and
we define $S\subseteq\{r\in\R\mid \rc_r=\rc\}$ arbitrary with $\abs{S}=k+1$.
Then
\begin{align*}
\sum_{r\in S} \rc_r &= \rc k + \rc > \rc k + \nfp - \rc k = \nfp,\\
\sum_{r\in \R\setminus S} \rc_r &= \nR_1 + \rc (\nR_\rc-k-1)\\
&\stackrel{\nR_\rc\geq k+\ell+1}{\geq} \nR_1 + \rc (k+\ell+1 -k -1) = \nR_1 + \rc\ell\\
&\stackrel{\nR_1\geq \rc -1}{\geq} \rc -1 +\rc\ell \stackrel{\rc-1\geq\nmp-\rc\ell}{\geq} \nmp-\rc\ell + \rc \ell\geq \nmp,
\end{align*}
i.e., the feasibility condition \cref{eq:def:feas} is satisfied and the instance is feasible.
\end{proof}

Remark that the condition $\nR_1\geq \rc-1$ in \cref{lem:feas:1-c} is tight:
let $c=3$ and $\nR_1< \rc-1$, i.e., let $\nR_1=\nR_2=1$. Then an instance with $\nfp=\nmp=2$ satisfies \cref{eq:feas:mixed}, however, there exists no feasible solution.

\cref{lem:feas:1-c} covers especially the case of wards with single and double rooms only.
Additionally, we can use it to derive a similar result for wards with rooms of even capacity.

\begin{lemma}
\label{lem:feas:2-4}
Consider a ward with room capacities $\rc_r\in\{2,2\rc\}$ with $c\in\N_{\geq 2}$ for all rooms $r\in\R$.
Let the number of female and male patients $\nfp,\nmp\in\N_0$,
be given.
If $\nR_2\geq \rc-1$, then the instance is feasible if and only if for every time period $t \in \T$ one of the two following conditions holds true.
\begin{enumerate}
    \item $\nfp$ and $\nmp$ are both even and the number of patients does not exceed the ward's total capacity, i.e.,
        $\nfp +\nmp \le \sum_{r \in \R}c_r$\label{itm:feas:2-4:a}
    \item the number of patients is strictly smaller than the ward's total capacity, i.e.,
        $\nfp +\nmp < \sum_{r \in \R}c_r$\label{itm:feas:2-4:b}
\end{enumerate}
\end{lemma}
\begin{proof}
First, assume both conditions are violated, then we have $\nfp +\nmp=\sum_{r \in \R}c_r$, and both $\nfp$ and $\nmp$ are odd.
Then no feasible assignment of patients to rooms exists.
Second, assume Condition \ref{itm:feas:2-4:a} holds true.
Then we construct an equivalent instance by dividing all $\rc_r$, $\nfp$, and $\nmp$ by $2$.
This instance is feasible according to \cref{lem:feas:1-c} and hence also the original one.

Third, assume Condition \ref{itm:feas:2-4:b} holds true and Condition \ref{itm:feas:2-4:a} does not.
Thus, we have $\nfp +\nmp < \sum_{r \in \R}c_r$ and either both $\nfp$ and $\nmp$ are odd or exactly one of them.
If both $\nfp$ and $\nmp$ are odd, it directly follows that $\nfp +\nmp\leq \sum_{r \in \R}c_r -2$.
We then construct an equivalent instance by increasing the number of female and male patients each by one, i.e., $\nfp':=\nfp+1$ and $\nmp':=\nmp+1$.
Then $\nfp'$ and $\nmp'$ are both even with $\nfp'+\nmp'\leq \sum_{r \in \R}c_r$.
This new instance is feasible according to Condition \ref{itm:feas:2-4:a} and hence also the original one.
We proceed analogously if either $\nfp$ or $\nmp$ is odd.
\end{proof}

\subsection{Maximum number of private patients in single rooms}\label{sec:dualBounds}
We define the problem of computing the maximum number $\smax_t$ of private patients who can get a room for themselves in time period $t$ as follows.

\begin{definition}[Private Patient Problem (\ppp)]
Let the total number of female and male patients $\nfp,\nmp\in\N_0$, the number of female and male private patients $\nprivf,\nprivm\in\N_0$, and room capacities $\rc_r\in\N$ for $r\in\R$ be given.
Do there exist four disjoint subsets $S_F\cup S_F^*\cup S_M\cup S_M^*\subseteq \R$ such that 
\begin{enumerate}
    \item all female patients are assigned to rooms $S_F\cup S_F^*$, and all patients assigned to rooms in $S_F^*$ are private patients and alone in their rooms, i.e.,
    \begin{equation}\label{eq:def:smax:a}
     \sum_{r\in S_F} \rc_r + \abs{S_F^*} \geq \nfp \quad\text{and}\quad \abs{S_F^*}\leq \nprivf,
    \end{equation}
    \item all male patients are assigned to rooms $S_M\cup S_M^*$, and all patients assigned to rooms in $S_M^*$ are private patients and alone in their rooms, i.e.,
    \begin{equation}\label{eq:def:smax:b}
     \sum_{r\in S_M} \rc_r + \abs{S_M^*} \geq \nmp \quad\text{and}\quad \abs{S_M^*}\leq \nprivm,
    \end{equation}
    \item the number of private patients who have a room to themselves is maximal, i.e.,
    \begin{equation}\label{eq:def:smax:c}
        \abs{S_F^*}+\abs{S_M^*} \quad\text{is maximal.}
    \end{equation}
\end{enumerate}
\end{definition}

We first take a look at the complexity of \ppp.
\begin{lemma}
\ppp is $\np$-hard and not approximable.
\end{lemma}
\begin{proof}
For $\nprivf=\nprivm=0$, \ppp is equivalent to the feasibility problem.
Hence, also \ppp is $\np$-complete.
Since the objective value in this case is $0$, \ppp is not approximable.
\end{proof}

However, \ppp can be solved in polynomial time if the ward has only single and double rooms.

\begin{lemma}
For feasible instances with $\rc_r\in\{1,2\}$, we can compute $\smax_t$ as follows.
Let 
\begin{align*}
    \alpha_t&:=\abs{\R}-\ceil*{\frac{\nfp-\nprivf}{2}} - \ceil*{\frac{\nmp-\nprivm}{2}},\\
    \beta^\mathrm{f}_t&:=\min\left\{\left(\nfp-\nprivf\right)\mod{2},~\nprivf\right\}\in\{0,1\}, \text{ and}\\
    \beta^\mathrm{m}_t&:=\min\left\{\left(\nmp-\nprivm\right)\mod{2},~\nprivm\right\}\in\{0,1\}.
\end{align*}
Then
\begin{align*}
    \smax_t = \begin{cases}
                    \abs{\priv(t)} & \text{if } \alpha_t \geq \abs{\priv(t)},\\
                    \abs{\priv(t)}-1 & \text{if } \alpha_t = \abs{\priv(t)}-1 \text{ and } \beta^\mathrm{f}_t=\beta^\mathrm{m}_t=1,\\
                    2\alpha_t + \beta^\mathrm{f}_t + \beta^\mathrm{m}_t - \abs{\priv(t)} & \text{otherwise}.
            \end{cases}
\end{align*}
\end{lemma}

\begin{proof}
For feasible instances with single and double rooms, we can treat all single rooms as double rooms since this does not affect the number of private patients who can get a room for themselves.
Therefore, let us consider a feasible instance of \ppp with only double rooms, i.e., equation \eqref{eq:feas:double} holds true.
We now have to assign at least
\[\ceil*{\frac{\nfp-\nprivf}{2}} +\ceil*{\frac{\nmp-\nprivm}{2}}
\]
rooms to non-private patients.
Since we aim to maximize the number of private patients who are alone in a room, we assign exactly that many rooms to non-private patients.
If the number of free remaining rooms $\alpha_t$ is greater or equal to the number of unassigned private patients, i.e.,
$\alpha_t \geq \abs{\priv(t)}$,
then every private patient can get a room for themselves, i.e., \[\smax_t=\abs{\priv}.\]

Otherwise, after assigning all non-private patients, we have $\alpha_t$ empty double rooms as well as potentially one ($\beta^\mathrm{f}_t$) free bed in a double-bed room where a non-private female patient is present which we can assign to a private female patient (if at least one is present), respectively for male patients ($\beta^\mathrm{m}_t$).
This results in a total of
\[\gamma_t:=2\alpha_t + \beta^\mathrm{f}_t + \beta^\mathrm{m}_t\]
available beds for private patients.
If $\beta^\mathrm{f}_t=0$ or $\beta^\mathrm{m}_t=0$ or $\alpha_t \leq \abs{\priv(t)}-2$, then
the difference of $\beta_t$ and the total number of private patients gives us the of number of empty beds, i.e., the number of private patients who can get a room for themselves because the potentially free beds in rooms with exactly one non-private patients will always be used in this case, i.e.,
\[
    \smax_t=\gamma_t- \abs{\priv(t)}.
\]
However, if both $\beta^\mathrm{f}_t=1$ and $\beta^\mathrm{m}_t=1$ and exactly $\abs{\priv(t)}=\alpha_t+1$ private patients need a room, then exactly one private patient will be placed in a room together with a non-private patient, i.e.,
\[\smax_t=\abs{\priv(t)}-1.\]
Overall, we achieve the stated formula for computing $\smax_t$.
\end{proof}

Knowing the maximum number $\smax_t$ of private patients who can get a single room in time period $t$ allows us to assess the trade-off between $\fpriv$ and $\ftrans$ or other objectives that occur in practice, e.g. hosting all patients who need immediate care.
Using the exact computation of $\smax_t$, we know that their sum over all time periods $t\in\T$ is a tight upper bound on the total objective value for $\fpriv$, i.e.,
\begin{equation} \label{L}
\fpriv\leq \smax :=\sum_{t \in \T} \smax_t.
\end{equation}
This bound can always be achieved as long as arbitrary many transfers may be used.

\section{Comparison of different IP-formulations}
\label{sec:ip:general}
In this section, we propose and compare different IP-formulations for PRA. 
Since we present multiple formulations for some of the conditions, we explain every constraint individually and then state for every IP which of the constraints are used.
To reduce the total number of IP variants, we first compare different formulations for minimizing the total number of transfers.
Second, we use the best performing LP formulation for minimizing transfers and then compare different extensions for incorporating single room requests of private patients.
Third, we compare different IP-formulations for maximizing $\fpriv$ without using transfers.

\subsection{Minimize transfers only}
\label{subsec:compstudy:mintransonly}
To model the assignment of patients to rooms as well as the minimization of transfers, we use the following binary variables:
\begin{align}
x_{prt}&=\begin{cases}
            1,  &\text{if patient }  p \text{ is assigned to room } r \text{ in time period } t,\\
            0,  &\text{otherwise,}
        \end{cases}\\
\delta_{prt}&=\begin{cases}
            1,  &\text{if patient } p \text{ is transferred from room } r \text{ to another room}\\ &\text{after time period } t\\
            0,  &\text{otherwise.}
        \end{cases}
\end{align}
We then model the total number of transfers as the sum of all variables $\delta$ together with all altered pre-fixed assignments
\begin{equation}\label{eq:xprt:transobj}
    \transobj = \sum_{t\in\T}\sum_{p\in\P(t)}\sum_{r\in\R} \delta_{prt} + \abs{\rpold}-\sum_{(r,p)\in\rpold} x_{pr1}.
\end{equation}

Regarding the constraints, we first ensure that all patients are assigned to rooms for every time period of their stay:
\begin{equation}
\label{eq:everypatientAroom}
\sum_{r\in\R} x_{prt} = 1   \quad \forall t\in\T, p\in\P(t).
\end{equation}
Second, we ensure that the room capacity is respected via
\begin{equation}\label{eq:capacity}
    \sum_{p\in\P(t)} x_{prt} \leq \rc_r   \quad \forall t\in\T, r\in\R.
\end{equation}
Third, 
to model sex separation, we introduce two additional sets of binary variables
\begin{align}
    g_{rt} &=\begin{cases}
            1,  &\text{if there is a female patient assigned to room } r \text{ in time}\\ &\text{period } t,\\
            0,  &\text{otherwise,}
        \end{cases}\label{var:grt}\\
    m_{rt} &=\begin{cases}
            1,  &\text{if there is a male patient assigned to room } r \text{ in time}\\ &\text{period } t,\\
            0,  &\text{otherwise.}\label{var:mrt}
        \end{cases}
\end{align}
We then the ensure sex separation via
\begin{align}
    x_{prt}&\leq g_{rt} \quad&& \forall t\in\T,\ p\in\fp(t),\ r\in\R,\label{eq:Sex:grt}\\
    x_{prt}&\leq m_{rt} && \forall t\in\T,\ p\in\mp(t),\ r\in\R,\label{eq:Sex:mrt}\\
    g_{rt}+m_{rt}&\leq 1 && \forall t\in\T,\ r\in\R.\label{eq:Sex:gm}
\end{align}
Using $m_{rt}\leq 1-g_{rt}$ we can remove variable $m_{rt}$ and replace constraints \cref{eq:Sex:mrt,eq:Sex:gm} with
\begin{equation}\label{eq:Sex:grt:only}
    x_{prt}\leq 1-g_{rt} \quad \forall t\in\T,\ p\in\mp(t),\ r\in\R.
\end{equation}
Instead of modeling capacity and sex separation constraints separately, we can also combine them and use
\begin{align}
    \sum_{p\in\fp(t)} x_{prt}&\leq \rc_r g_{rt} \quad&& \forall t\in\T,\ r\in\R,\label{eq:Capacity+Sex:grt}\\
    \sum_{p\in\mp(t)} x_{prt}&\leq \rc_r m_{rt} && \forall t\in\T,\ r\in\R,\label{eq:Capacity+Sex:mrt}
\end{align}
instead of \cref{eq:capacity,eq:Sex:grt,eq:Sex:mrt}.
Or, if we omit variable $m_{rt}$, we use 
\begin{equation}
    \sum_{p\in\mp(t)} x_{prt}\leq \rc_r (1-g_{rt}) \quad \forall t\in\T,\ r\in\R,\label{eq:Capacity+Sex:grt:only}
\end{equation}
instead of of \cref{eq:Sex:gm,eq:Capacity+Sex:mrt}.
Fourth, we count the transfers via
\begin{equation}\label{eq:xprt:delta}
    x_{prt} - x_{pr(t+1)} \leq \delta_{prt} \quad \forall r\in\R,\ p\in\P, \arr_p\leq t < \dis_p-1.
\end{equation}
We compare the performance of the following four IP-formulations to
investigate the usage of variables $m_{rt}$, as well as the integration of capacity and sex separation constraints.

\begin{enumerate}[(A)]
    \item  $\min \transobj$ s.t. \cref{eq:everypatientAroom,eq:capacity,eq:Sex:grt,eq:Sex:mrt,eq:Sex:gm,eq:xprt:delta}   \label{IP:A}
    \item  $\min \transobj$ s.t. \cref{eq:everypatientAroom,eq:capacity,eq:Sex:grt,eq:Sex:grt:only,eq:xprt:delta}
    \item  $\min \transobj$ s.t. \cref{eq:everypatientAroom,eq:Sex:gm,eq:Capacity+Sex:grt,eq:Capacity+Sex:mrt,eq:xprt:delta}
    \item  $\min \transobj$ s.t. \cref{eq:everypatientAroom,eq:Capacity+Sex:grt,eq:Capacity+Sex:grt:only,eq:xprt:delta} \\ \label{IP:D}
\end{enumerate}

\label{sec:transfersonlystudy}
All IPs were implemented in \emph{python} 3.10.4 and solved using \emph{Gurobi 10.0.0}.
All simulations were done on the \href{https://www.itc.rwth-aachen.de/cms/IT-Center/Services/Servicekatalog/Hochleistungsrechnen/~hisv/RWTH-Compute-Cluster/?lidx=1}{RWTH High Performance Computing Cluster} using CLAIX-2018-MPI with Intel Xeon Platinum 8160 Processors “SkyLake” (2.1 GHz, 32 CPUs per task, 3.9 GB per CPU).
The code can be found at \gitlink.

For testing, we used 62 real-world instances provided by the \emph{RWTH Aachen University Hospital (UKA)}, each spanning a whole year, and a time limit of 12h.
We performed consistency checks on the patient data ensuring valid input data: patients with missing information on arrival or discharge and patients with $\arr_p=\dis_p$ were dropped from the data and for patients whose registration was noted after their arrival, we set the registration date to the arrival date.
All instances together still contain more than $53.000$ patient stays.
For every instance, the number of rooms and their capacities are given as well as the patients' arrival, departure, and registration dates, their sex, unique Patient-ID and information on the insurance status.
Note that the data is subject to non-disclosure and as such is not provided together with the code.

The results of comparing "transfers only" formulations are depicted in \cref{fig:A-D:runtime:43200}. They show that the integration of capacity and sex separation constraints decreases computation time. Similarly, removing the variable $m_{rt}$ also decreases computation time.
In general, instances were either solved to optimality with objective value 0 or resulted in a MIPGap of 100\% after 12 hours.
\begin{figure}[tb!]
    \centering
    \includegraphics[width=0.7\textwidth]{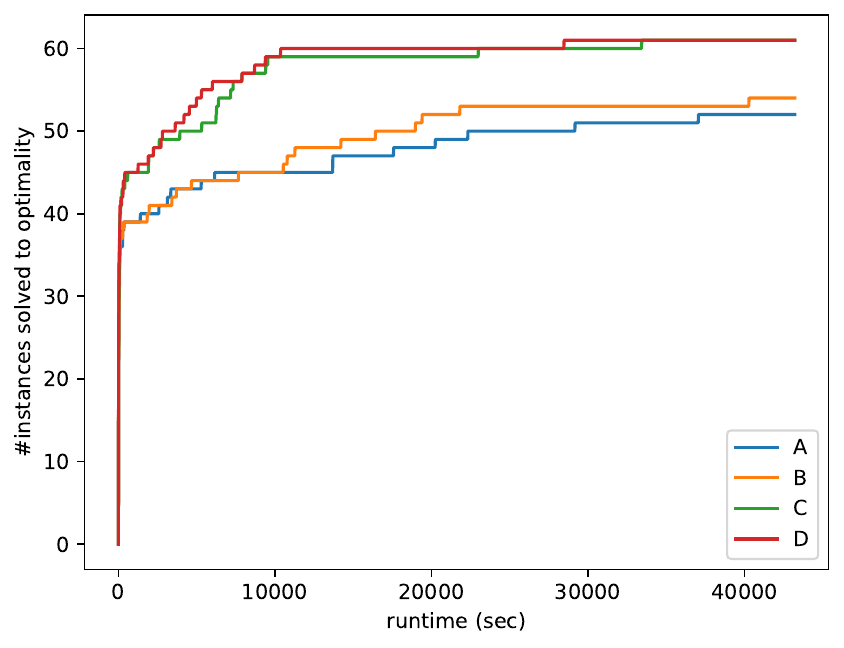}
    \caption{Comparison of IPs A - D using 62 real-life instances, after 12 h 61 instances were solved to optimality by IPs C and D with objective value 0}
    \label{fig:A-D:runtime:43200}
\end{figure}

\subsection{Integration of single room constraints}
\label{subsec:IP:single-room}
To incorporate single room request,
we define binary variables encoding whether a private patient gets a single room via
\begin{equation}
s_{prt} =\begin{cases}
        1,  &\text{if } p \text{ is alone in room } r \text{ in time period } t,\\
        0,  &\text{otherwise.}
    \end{cases}
\label{var:sprt}
\end{equation}
Thus, the total number of time periods that private patients are assigned to single rooms is given by
\begin{equation}
\label{eq:sprt:privobj}
    \privobj = \sum_{t\in\T}\sum_{p\in\priv(t)}\sum_{r\in\R} s_{prt}.
\end{equation}
Then, we can model the single room constraints via
\begin{align}
    s_{prt} &\leq x_{prt}   &&\forall t\in\T,\ p\in\priv(t),\ r\in\R,\label{eq:single:s}\\
    \rc_r s_{prt} + \sum_{q\in\P(t)\setminus\{p\}} x_{qrt} &\leq \rc_r   &&\forall t\in\T,\ p\in\priv(t),\ r\in\R.\label{eq:single:c}
\end{align}
Alternatively to \cref{eq:single:c}, we can also integrate the single room constraints with the sex separation and capacity constraints \cref{eq:Capacity+Sex:grt,eq:Capacity+Sex:mrt} via
\begin{align}
    \sum_{p\in\fp(t)} x_{prt} + \sum_{p\in\fp\cap\priv(t)} (\rc_r -1)s_{prt} &\leq \rc_r g_{rt} &&\forall t\in\T,\ r\in\R\label{eq:single:grt}\\
    \sum_{p\in\mp(t)} x_{prt} + \sum_{p\in\mp\cap\priv(t)} (\rc_r -1)s_{prt} &\leq \rc_r (1-g_{rt}) &&\forall t\in\T,\ r\in\R\label{eq:single:grt:only}.
\end{align}

We compare the performance of the following LP-formulations that integrate single room requests based on the previous results
\begin{enumerate}[(A)]
\setcounter{enumi}{4}
    \item $\max~(-\transobj,\privobj)$ s.t. constraints of \ref{IP:D}, \cref{eq:single:s,eq:single:c}\label{IP:E}
    \item $\max~(\privobj,-\transobj)$ s.t. constraints of \ref{IP:D}, \cref{eq:single:s,eq:single:c}
\stepcounter{enumi}
    \item $\max~(-\transobj,\privobj)$ s.t. \cref{eq:everypatientAroom,eq:xprt:delta,eq:single:s,eq:single:grt,eq:single:grt:only}\label{IP:H}
    \item $\max~(\privobj,-\transobj)$ s.t. \cref{eq:everypatientAroom,eq:xprt:delta,eq:single:s,eq:single:grt,eq:single:grt:only}\label{IP:I}.
\end{enumerate}
The objectives' order determines their priority in optimisation, i.e., $\max (-\transobj,\privobj)$ means that first $\transobj$ is minimised and then $\privobj$ is maximised.

The formulations for IPs \ref{IP:E} to \ref{IP:I} were evaluated on the same computational setup as in \cref{sec:transfersonlystudy}. The results are given in \cref{fig:E-J:runtime:43200}.
\begin{figure}[bt!]
    \centering
    \includegraphics[width=0.7\textwidth]{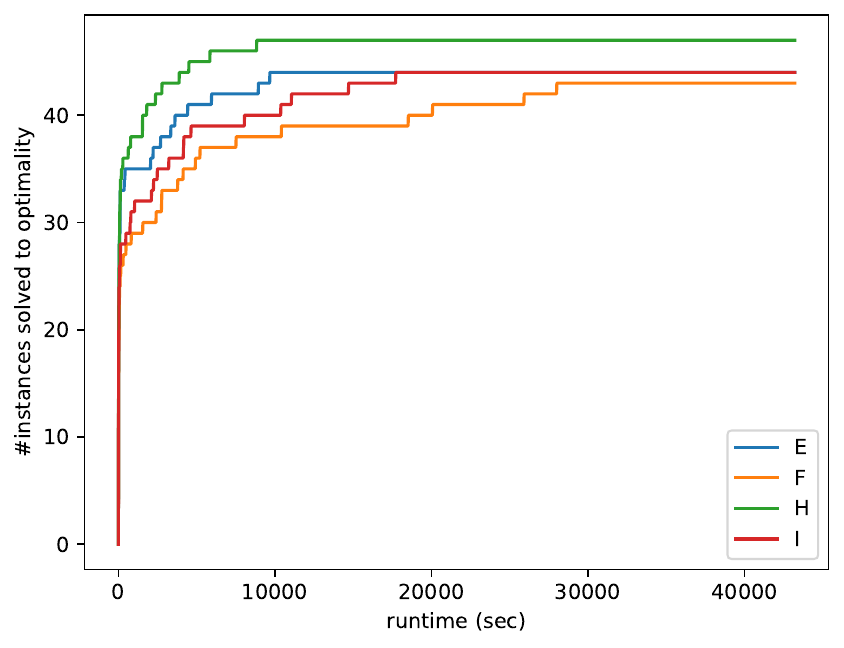}
    \caption{Comparison of IPs E - H using 62 real-life instances, maximum runtime 12h}
    \label{fig:E-J:runtime:43200}
\end{figure}
We see that the decisive factor is not the set of constraints but the objective function.
Minimizing the number of transfers first $\max (-\transobj,\privobj)$ performs significantly better than maximizing the private patients first.
However, it is noticeable that the second set of constraints performs overall better than the first set. 
We further observed that, when optimizing $\privobj$, the solver frequently finds an optimal solution quickly but then requires extended time to prove optimality.
Therefore, we use our combinatorial insights the help the solver prove optimality in this case.

If maximizing $\privobj$ has highest priority, we can use the combinatorial insights from \cref{sec:combi}
and fix the number of private patients in single rooms for time period $t$ to $\smax_t$, i.e.,
\begin{equation}\label{eq:fix:smax_t}
    \sum_{p\in\priv(t)}\sum_{r\in\R} s_{prt} \geq \smax_t \quad \forall t\in\T
\end{equation}
instead of using the biobjective approach.
Hence, we also evaluate the resulting IP
\begin{enumerate}[(A)]
\setcounter{enumi}{10}
    \item $\min \transobj$ s.t. constraints of \ref{IP:H}. \cref{eq:fix:smax_t}\label{model:fix:smax}\label{IP:KK}
\end{enumerate}

We compare the respective IP's performance to the one of \ref{IP:H} and \ref{IP:I}.
\begin{figure}[tb!]
    \centering
    \includegraphics[width=0.7\textwidth]{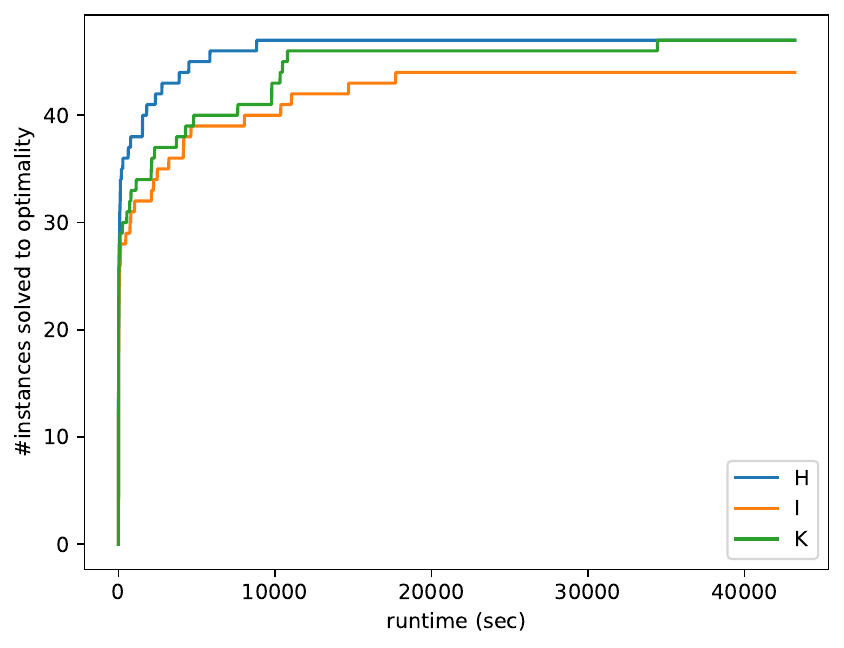}
    \caption{Performance of IP K using 62 real-life instances, maximum runtime 12h}
    \label{fig:KL:runtime:43200}
\end{figure}
\cref{fig:KL:runtime:43200} shows that IP \ref{IP:KK} clearly outperforms IP \ref{IP:I},
however, its performance is not as good as the one of IP \ref{IP:H}.

\subsection{IP-formulation without transfers}\label{sec:ip:notransfer}
The objective values computed in our computational experiments in \cref{subsec:compstudy:mintransonly,subsec:IP:single-room} showed that in many instances no transfers are necessary throughout the entire planning period of one year.
Therefore, we propose and compare in this section IP formulations where transfers are prohibited by construction while $\privobj$ is maximized.
We compare again three different levels of constraint integration.
Furthermore, we evaluate whether it is faster to solve the optimization problem with objective function $\privobj$ or the feasibility problem where $\privobj=\smax$ is fixed.

We use binary variables
\begin{align}
x_{pr}&=\begin{cases}
            1,  &\text{if patient }  p \text{ is assigned to room } r \text{ for their stay},\\
            0,  &\text{otherwise,}
        \end{cases}
\end{align}
to model the assignment of patients to rooms togehter with the previously introduced variables $s_{prt}$ as in \eqref{var:sprt}, and variables $g_{rt}$ as in \eqref{var:grt}.

Regarding the constraints, we first ensure that all patients are assigned to rooms in every time period of their stay: 
\begin{equation}
\label{eq:xpr:everypatientAroom}
    \sum_{r\in\R} x_{pr} = 1   \quad \forall p\in\P.\\
\end{equation}
Second, we ensure that the room capacity is respected via
\begin{equation}\label{eq:xpr:capacity}
    \sum_{p\in\P(t)} x_{pr} \leq \rc_r   \quad \forall t\in\T, r\in\R.
\end{equation}
Third, we ensure sex separation via
\begin{align}
    x_{pr}&\leq g_{rt} \quad&& \forall t\in\T,\ p\in\fp(t),\ r\in\R,\label{eq:xpr:Sex:grt}\\
    x_{pr}&\leq (1-g_{rt}) && \forall t\in\T,\ p\in\mp(t),\ r\in\R\label{eq:xpr:Sex:grt:only}.
\end{align}
Instead of modeling capacity and sex separation constraints separately, we can also combine them and use
\begin{align}
    \sum_{p\in\fp(t)} x_{pr}&\leq \rc_r g_{rt} \quad&& \forall t\in\T,\ r\in\R,\label{eq:xpr:Capacity+Sex:grt}\\
    \sum_{p\in\mp(t)} x_{pr}&\leq \rc_r (1-g_{rt}) && \forall t\in\T,\ r\in\R,\label{eq:xpr:Capacity+Sex:grt:only}
\end{align}
instead of \cref{eq:xpr:capacity,eq:xpr:Sex:grt,eq:xpr:Sex:grt:only}.
Fourth, we model the single room constraints via
\begin{align}
    s_{prt} &\leq x_{pr}   &&\forall t\in\T,\ p\in\priv(t),\ r\in\R,\label{eq:xpr:single:s}\\
    \rc_r s_{prt} + \sum_{q\in\P(t)\setminus\{p\}} x_{qr} &\leq \rc_r   &&\forall t\in\T,\ p\in\priv(t),\ r\in\R.\label{eq:xpr:single:c}
\end{align}
Alternatively to \cref{eq:xpr:single:c}, we can also integrate the single room constraints with the sex separation and capacity constraints \cref{eq:xpr:Capacity+Sex:grt,eq:xpr:Capacity+Sex:grt:only} via
\begin{align}
    \sum_{p\in\fp(t)} x_{pr} + \sum_{p\in\fp\cap\priv(t)} (\rc_r -1)s_{prt} &\leq \rc_r g_{rt} &&\forall t\in\T,\ r\in\R\label{eq:xpr:single:grt},\\
    \sum_{p\in\mp(t)} x_{pr} + \sum_{p\in\mp\cap\priv(t)} (\rc_r -1)s_{prt} &\leq \rc_r (1-g_{rt}) &&\forall t\in\T,\ r\in\R \label{eq:xpr:single:grt:only}.
\end{align}
Last, we ensure that any pre-fixed assignments are respected:
\begin{equation}\label{eq:xpr:prefix}
    x_{pr} = 1 \quad \forall (p,r)\in\rpold.
\end{equation}

We then compare the following IP-formulations to find the best performing constraint set.
\begin{enumerate}[(A)]
\setcounter{enumi}{12}
    \item $\max \privobj$ s.t. \cref{eq:xpr:everypatientAroom,eq:xpr:capacity,eq:xpr:Sex:grt,eq:xpr:Sex:grt:only,eq:xpr:single:s,eq:xpr:single:c,eq:xpr:prefix}\label{IP:M}
    \item $\max \privobj$ s.t. \cref{eq:xpr:prefix,eq:xpr:everypatientAroom,eq:xpr:Capacity+Sex:grt,eq:xpr:Capacity+Sex:grt:only,eq:xpr:single:c,eq:xpr:single:s}\label{IP:N}
    \item $\max \privobj$ s.t. \cref{eq:xpr:everypatientAroom,eq:xpr:single:s,eq:xpr:single:grt,eq:xpr:single:grt:only,eq:xpr:prefix}\label{IP:O}
    \item $\max 0$ s.t. constraints of \ref{IP:O}, \cref{eq:fix:smax_t}\label{IP:P}\\
\end{enumerate}

The formulations for IPs \ref{IP:M} to \ref{IP:P} were evaluated on the same computational setup as in \cref{sec:transfersonlystudy}.
The results show the dominance of IP \ref{IP:P} over the other IPs, cfg. \cref{fig:MNOP:runtime:43200}.
\begin{figure}[tb!]
    \centering
    \includegraphics[width=0.7\textwidth]{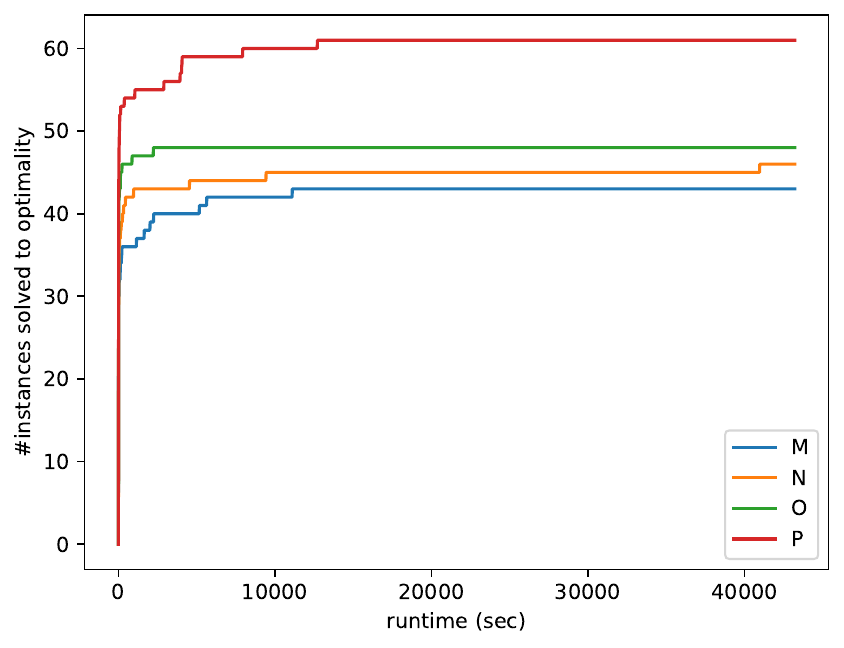}
    \caption{Comparison of IPs \ref{IP:M}-\ref{IP:P} using 62 real-life instances}
    \label{fig:MNOP:runtime:43200}
\end{figure}
However, it strongly depends on the use case whether IP \ref{IP:P} is the best one to use as, naturally, it is feasible in fewer instances than IP \ref{IP:O}.
With our real-life instances, \ref{IP:P} was feasible in $72.5$\% whereas \ref{IP:O} was feasible in $97.75\%$.
However, due to the fast runtime of IP \ref{IP:P}, cfg. \cref{fig:MNOP:runtime:120}, it may be worthwhile to check first whether IP \ref{IP:P} is feasible before switching to \ref{IP:O}.
\begin{figure}[h!]
    \centering
    \includegraphics[width=0.7\textwidth]{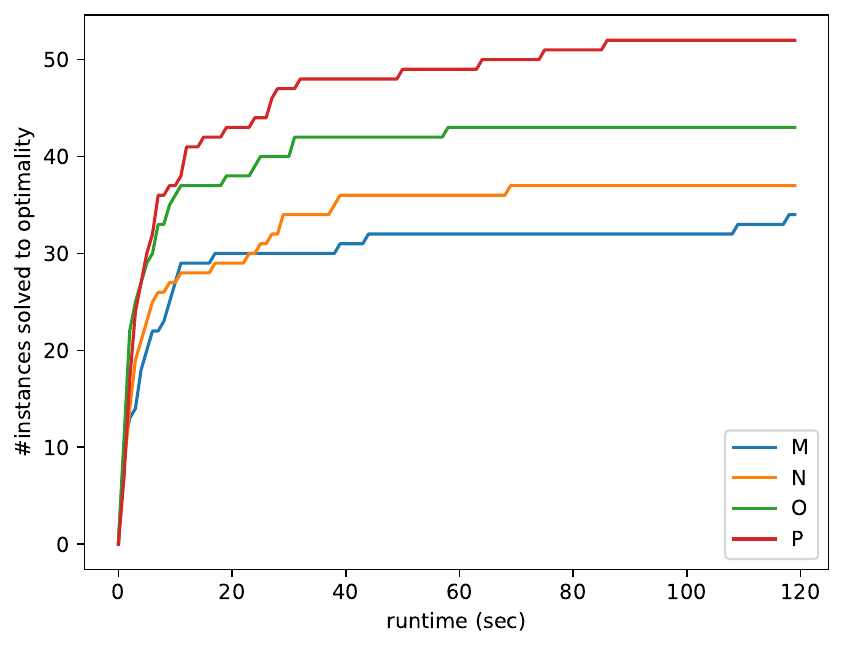}
    \caption{IP \ref{IP:P} solves 52 instances in $<100$ sec}
    \label{fig:MNOP:runtime:120}
\end{figure}

\section{Dynamic PRA}\label{sec:modelcombi}
As Dynamic PRA, we understand PRA with a rolling time horizon similar to the definition in \cite{Ouelhadj2008}.
As rescheduling is frequently done in practice, this approach relates more closely to the real life problem than the static version.
In this section, we describe how we combine four IP models and our combinatorial insights to efficiently solve the dynamic PRA by exploiting the IPs' different run-times.

Here, for every patient we are also given a registration time period so that the set $\P$ of all (known) patients is updated each time period.
For every time period $t\in\T$, all known patients, i.e., patients whose registration dates are before or equal to $t$, are assigned to rooms.
All room assignments of the current time period are then stored in the set $\rpold$.
We assume that $\rpold$ does not contain irrelevant data, i.e., discharged patients are deleted immediatly to ensure the correct computation of $\transobj$.
Hence, $\rpold$ is updated after every iteration just like the patient set $\P$.

The iterative nature of the dynamic PRA allows us to introduce a variant of IPs \ref{IP:P} and \ref{IP:O} where transfers are not entirely forbidden, but only changes to the current room assignment are allowed.
We call this concept \textit{same-day transfers} and formulate it as
\begin{enumerate}[(A*)]
\setcounter{enumi}{14}
    \item $\max~(\privobj,\sum_{(r,p)\in\rpold} x_{pr}$) s.t. \cref{eq:xpr:everypatientAroom,eq:xpr:single:s,eq:xpr:single:grt,eq:xpr:single:grt:only}\label{IP:Ostar}
    \item $\max~\sum_{(r,p)\in\rpold} x_{pr}$ s.t. constraints of \ref{IP:Ostar}, \cref{eq:fix:smax_t}.\label{IP:Pstar}
\end{enumerate}

For our algorithm, we combine the IPs \ref{IP:H}, \ref{IP:Ostar}, \ref{IP:P}, and \ref{IP:Pstar} and our combinatorial insights as follows.
First, we check combinatorially whether the instance is feasible
since we observed that the combinatorial feasibility check is significantly faster than building a respective IP (using gurobipy), not to mention solving it.
Second, we use the no-transfers formulation IP \ref{IP:P}.
Note that here, we make use of our second combinatorial insight, i.e., the computation of $\smax$.
If IP \ref{IP:P} is infeasible, we solve the instance again using the same-day transfer formulation IP~\ref{IP:Pstar}.
If IP~\ref{IP:Pstar} is also infeasible, we use IP \ref{IP:Ostar} maximising the number of private patients who get their own room while minimising the number of transfers in the first time period.
If again, no feasible solution for \ref{IP:Ostar} is found within $20$ seconds, we solve the instance using IP \ref{IP:H} which allows arbitrary many transfers and is therefore always feasible.

After successful computation, we fix all patient-room assignments for patients that are in hospital in the current time period by adding them to set $\rpold$ while removing outdated ones.
We then update the patient set and continue analogously with the next time period.
A visualization of this algorithm is provided in \cref{fig:modelcombi:dyn}.
\begin{figure}[tb!]
    \centering
    \scalebox{0.5}{
        \begin{tikzpicture}[node distance=2cm,every label/.style={align=left}]
    
    \tikzstyle{startstop} = [rectangle, rounded corners, minimum width=2.5cm, minimum height=1cm,text centered, draw=black, fill=RWTHred!50]

    \tikzstyle{process} = [rectangle, minimum width=1.5cm, minimum height=1cm, text centered, draw=black, fill=RWTHorange]
    
    \tikzstyle{decision} = [diamond, minimum width=2cm, minimum height=2cm, text centered, draw=black, fill=RWTHmaygreen]
    
    \tikzstyle{YN} = [minimum width=0cm, minimum height=0cm, text centered]

    \tikzstyle{arrow} = [thick,->,>=stealth]
        
    \node (initialisation) [startstop] {Initialisation};
    
    \node (feasibility) [decision, below of=initialisation, xshift=-1cm] {Feasible?};
    \node (no0) [YN,below of=feasibility, yshift=1.6cm,xshift=1.4cm] {No};
    \node (yes0d) [YN,below of=feasibility, yshift=1cm,xshift=0.6cm] {Yes};
     \draw [arrow] (initialisation) -- (feasibility);
    
    \node (N) [process, below of=feasibility] {\Large\ref{IP:P}};
    \node (N_eval) [decision, right of=N, xshift=1cm] {Feasible?};
    \node (Inc) [decision, below of=N_eval,yshift=-1cm] {\large$t=\nT?$};
    
    \draw [arrow] (feasibility) -- (N);
    \draw [arrow] (N) -- (N_eval);
    \draw [arrow] (N_eval) -- (Inc);
    
    \node (no1) [YN,below of=N_eval, yshift=1.6cm,xshift=1.4cm] {No};
    \node (yes1) [YN,below of=N_eval, yshift=0.5cm,xshift=0.4cm] {Yes};
    
    \node (no3) [YN,below of=Inc, yshift=1.6cm,xshift=-1.4cm] {No};
    \node (yes3) [YN,below of=Inc, yshift=0.5cm,xshift=0.4cm] {Yes};
    
    \node (C) [process, right of=N_eval, xshift=1cm] {\Large\ref{IP:Pstar}};
    \node (C_eval) [decision, right of=C, xshift=1cm] {Feasible?};
    
    \draw [arrow] (N_eval) -- (C);
    \draw [arrow] (C) -- (C_eval);
    \draw [arrow] (C_eval) |- (Inc);
    
    \node (no2) [YN,below of=C_eval, yshift=1.6cm,xshift=1.4cm] {No};
    \node (yes2) [YN,below of=C_eval, yshift=0.5cm,xshift=0.4cm] {Yes};
    
    \node (O) [process, right of=C_eval, xshift=1cm] {\Large\ref{IP:Ostar}};
    \node (O_eval) [decision, right of=O, xshift=1cm] {Feasible?};
    
    \draw [arrow] (C_eval) -- (O);
    \draw [arrow] (O) -- (O_eval);
    \draw [arrow] (O_eval) |- (Inc);
    
    \node (noO2) [YN,below of=O_eval, yshift=1.6cm,xshift=1.4cm] {No};
    \node (yesO2) [YN,below of=O_eval, yshift=0.5cm,xshift=0.4cm] {Yes};
    
    \node (E) [process, right of=O_eval, xshift=1cm] {\Large\ref{IP:H}};
    \node (E_eval) [decision, right of=E, xshift=1cm] {Feasible?};
    \node (INF) [startstop, right of=E_eval, yshift=-4.8cm] {Terminate};
    
    \node (no4) [YN,below of=E_eval, yshift=1.6cm,xshift=1.4cm] {No};
    \node (yes4) [YN,below of=E_eval, yshift=0.5cm,xshift=0.4cm] {Yes};
    
    \draw [arrow] (O_eval) -- (E);
    \draw [arrow] (E) -- (E_eval);
    \draw [arrow] (E_eval) |- (Inc);
    \draw [arrow] (E_eval) -| (INF);
    \draw [arrow] (Inc) |- (INF);
    \draw [arrow] (feasibility) -| (INF);
    
    \draw [arrow] (Inc) -| (N);
    \draw[dashed] (16,0) rectangle (22,-6);

    \node (IncLabel1) [YN,below of=Inc, yshift=1.6cm,xshift=-3.2cm] {Update $\P,\rpold$};
    \node (IncLabel2) [YN,below of=Inc, yshift=1.2cm,xshift=-3.2cm] {$t \pluseq 1$};

\end{tikzpicture}
    }
    \caption{Algorithm for dynamic PRA, the dotted part was never reached in our computational study}
    \label{fig:modelcombi:dyn}
\end{figure}
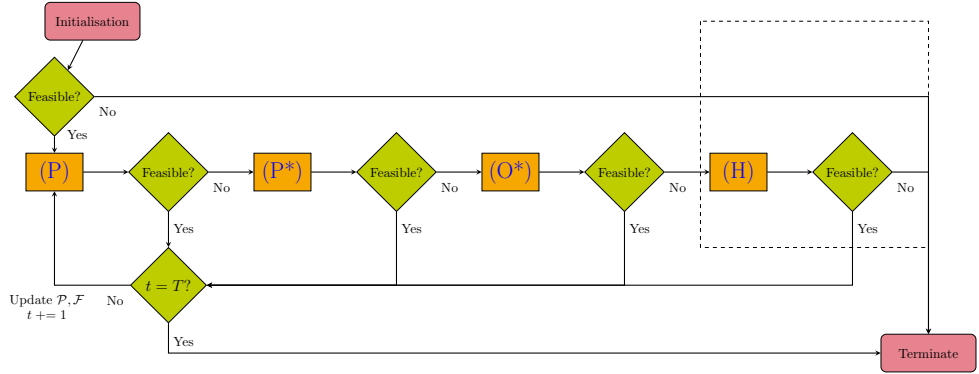

We evaluate our algorithm again on 62 real-world instances spanning a whole year.
As a result we get that all instances use 365 iterations of the algorithm and all are solved within less than 600 seconds per year, cfg.~\cref{fig:dyn:runtime}.
\begin{figure}[tbh!]
    \centering
    \includegraphics[width=0.7\textwidth]{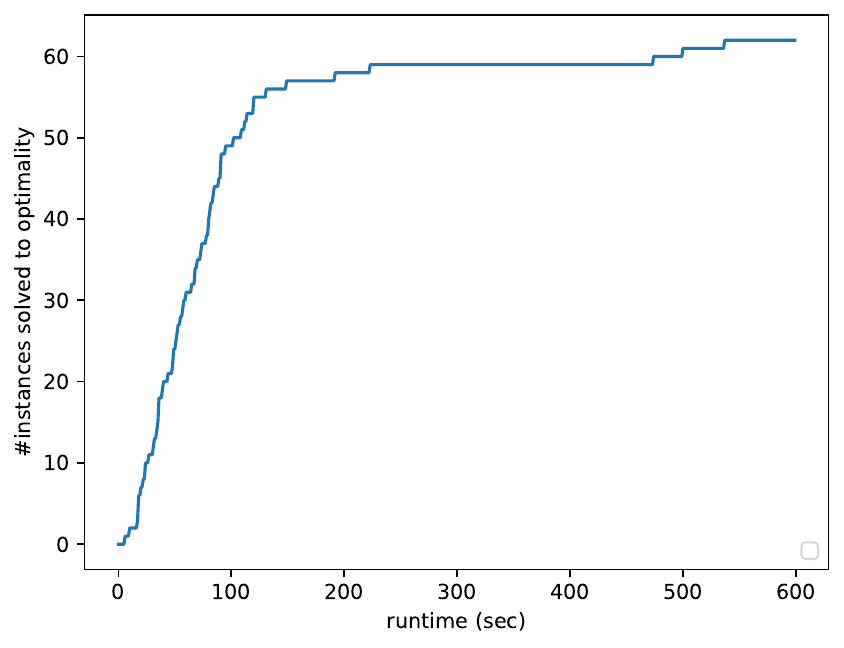}
    \caption{Runtime of algorithm for dynamic PRA with $\nT=365$}
    \label{fig:dyn:runtime}
\end{figure}
For application purposes however, the runtime per iteration is more interesting than the total runtime of $365$ iterations.
Therefore, we report in \cref{fig:dyn:runtime:perIter} the runtime of all $62\cdot 365=22630$ iterations individually.
The results show that all but three iterations are solved within less than 15 seconds,  cfg.~\cref{fig:dyn:runtime:box},
and more than 95\% of all iterations are solved within less a second, cfg.~\cref{fig:dyn:runtime:log-box}.
\begin{figure}[tb!]
\begin{subfigure}[b]{0.5\textwidth}
    \centering
    \includegraphics[width=\textwidth]{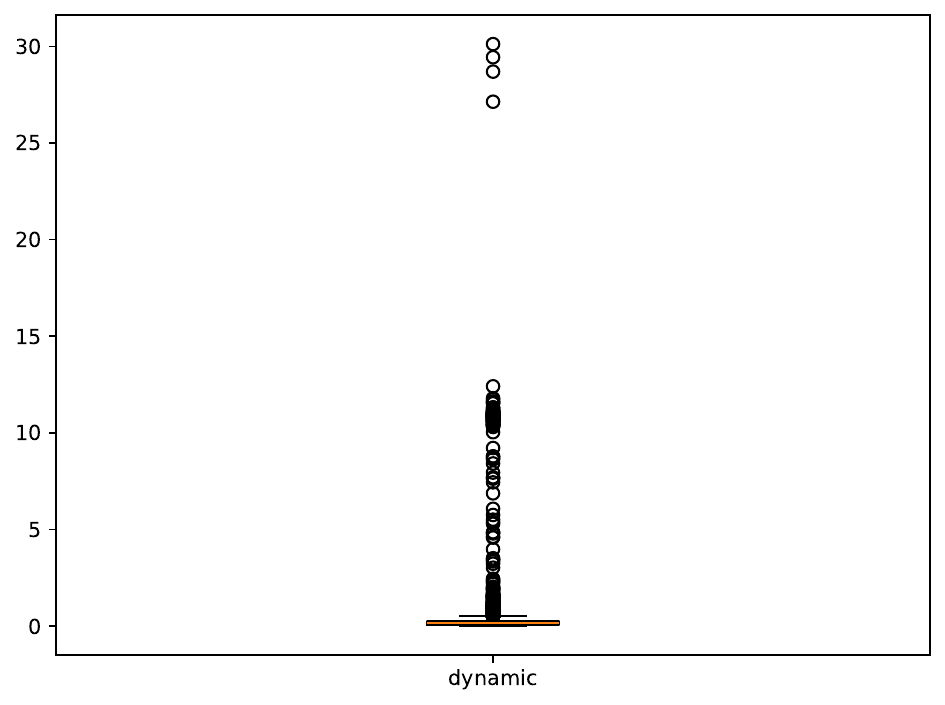}
    \caption{normal axis}
    \label{fig:dyn:runtime:box}
\end{subfigure}
\begin{subfigure}[b]{0.5\textwidth}
    \centering
    \includegraphics[width=\textwidth]{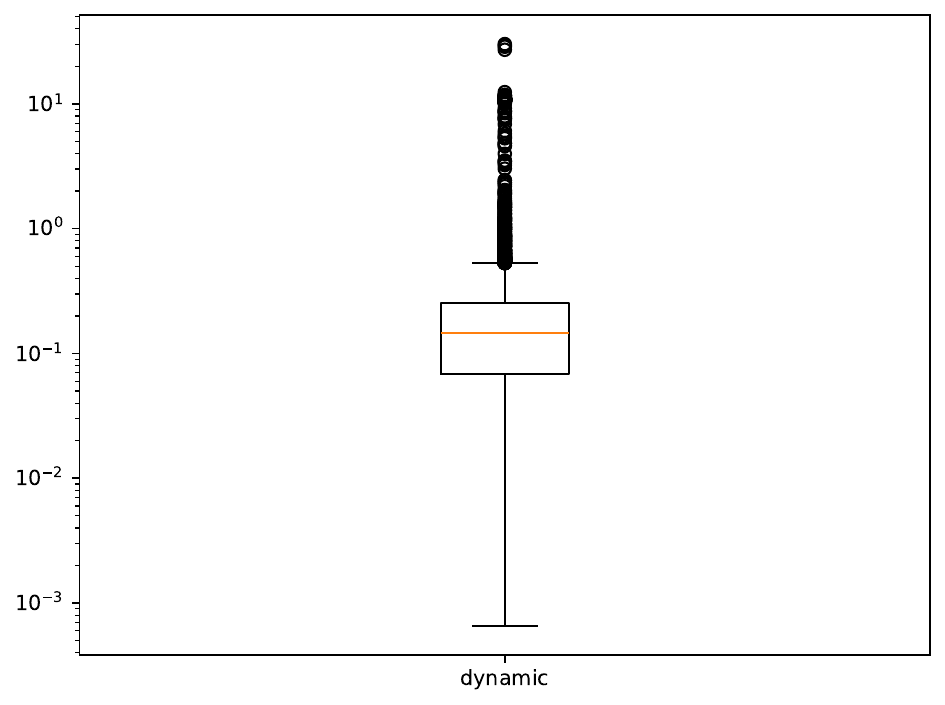}
    \caption{logscaled axis}
    \label{fig:dyn:runtime:log-box}
\end{subfigure}
\caption{Runtime per iteration of the algorithm for dynamic PRA}
\label{fig:dyn:runtime:perIter}
\end{figure}

Although our dynamic algorithm is a heuristic and, thus, does not guarantee a certain solution quality, we can assess a solution's quality using our combinatorial insights:
For $44$ instances the optimal value $\fpriv=\smax$ was achieved.
For $12$ instances, we achieve $\fpriv\geq 0,9885\smax$ and for one $\fpriv = 0,946\smax$.
The remaining $5$ instances have no private patients.
The high quality of our solutions w.r.t. $\fpriv$ is especially remarkable since in $26$ of them no transfers are needed, $28$ use between $1$ and $27$ transfers, and $8$ between 28 and 80 transfers.

\section{Future work}\label{sec:further_research}
We close this paper by pointing out multiple not yet fully explored aspects of PRA to inspire future research.
We give an overview over possible modelling extensions for our definition of PRA.
Where possible, we provide first experimental computational results and point out promising areas for further research.

\subsection{Scaling to multiple wards}\label{subsec:multiplewards}
Similar wards within the same specialty can be  planned jointly. Initial computational testing showed that, in this case, the run-time scales linearly up to a 150 rooms over a planning horizon of 365 time periods. 

The proposed IP modelling approach can also be extended to manage multiple dissimilar wards.
For this, new constraints must be added to model which patient can be assigned to which ward.
We evaluated this both for single specialties with up to $9$ normal and $2$ intensive care wards, and for the full RWTH Aachen University Hospital with $800$ rooms and $53.000$ patients (again over a planning horizon of $365$ time periods).
The run-time for the full hospital averages to about $2$s per iteration, with a larger variation than for single wards, i.e, some days requiring more than $10$s for an initial feasible solution.
However, in our modelling approach patient-assignment feasibility was only based on past patient-stay data, without consulting with medical professionals.
Thus, further research first needs to identify suitable metrics for patient-ward suitability and then include these as an objective component.

\subsection{Patient conflicts}
Due to medical or social reasons
there may be pairs of patients who cannot share a room, e.g., two patients with asthma or one woman whow just gave birth and one who lost the child.
Such so-called patient conflicts can easily be integrated into all our proposed IP-formulations by adding conflict constraints of assigning weights to patients.
Since we do not have any real data about patient conflicts, we experimented with a small number of randomly generated conflicts.
In our setting, this had neither an effect on the runtime nor the objective value.
However, in theory, a large number of conflicts may render an instance infeasible.
In the future, we will further investigate what conflicts occur in reality and constitutes their effect on runtime and solution quality.

\subsection{Patient preferences}
If more than one patient is assigned to a room, assigning suitable room-mates also constitutes a further goal \cite{roommates2}.
Specifically, patient combinations exist that may be beneficial both for patients and staff.
For example, it is known that patients recover faster if they feel comfortable, therefore, a room-mate whom they can relate to may be beneficial \cite{roommates1,roommates2}.
Or, if an international patient is not fluent in the local language, then it is beneficial for both patient and staff if the roommate can translate.
First computational experiments with IP-formulations showed that incorporating inter-patient preferences into the IP models leads to a significant increase in run-time.
Developing an efficient way to integrate the choice of suitable room-mates remains ongoing research.

\subsection{Accompanying person}
Some patients are entitled to bring an accompanying person with them to the hospital.
If the accompanying person occupies a normal patient bed,
this can easily be integrated into all our proposed IP-formulations by adding weights to patients and/or not implementing assignment variables for single rooms for the respective patients.
If the accompanying person sleeps on an additional roll-in bed and does not occupy a patient bed, it depends on the hospital's policy whether it is, e.g., desirable to avoid assigning multiple patients with an accompanying person to the same room or whether sex separation also needs to be respected for the accompanying person.
It is still ongoing research to determine the decisive criteria currently in use for this task.

\subsection{Uncertainty}
Considering uncertainty is essential to ensure real-world applicability and validity of results.
By using a dynamic time horizon, we already integrated the uncertain arrival of emergency patients.
A second and equally relevant factor, however, is the uncertainty in the length of stay.
It is easily possible to update a patient's planned discharge date in every iteration of our algorithm for dynamic PRA.
If a patient's stay is prolonged, the feasibility check should then be repeated for the affected time periods.
It is still an open question to assess the consequences of such updates on the solution quality.
For wards with a high uncertainty in patient length of stay, it also might be better to integrate the uncertainty more directly in the algorithm to compute robust solutions.
It is however also yet undefined what a \emph{robust} solution in this context means.

\section{Conclusion}\label{sec:conclusion}
In this work, we presented new combinatorial insights for the patient-to-room assignment problem with regard to feasibility and the assignment of private patients to single rooms.
We provided closed formulas to check an instance's feasibility and to compute the maximum number of single-room requests that can be fulfilled.
The computation time of those formulas is only a fraction of the time needed to build a corresponding IP.
This is of special interest, e.g., in the context of appointment scheduling in hospitals.

We further explored the performance of different IP-formulations.
One of our key insights here is the significant performance gap between objectives $\ftrans$ and $\fpriv$ which needs to be taken into account when designing IP-formulations.
Using all our insights, we proposed a fast IP-based solution approach that obtains high quality solutions which
showcases the benefits of combinatorial insights for developing solution approaches.
Our computational study
showed that even though PRA is $\mathcal{NP}$-hard, the PRA problem can be solved to optimality or at least close to optimality for realistically sized instances in reasonable time.

Finally, we elaborated on numerous possibilities for future work.

\section{Acknowledgements}
We thank the team at RWTH Aachen University hospital for their support.
This work was partially supported by the Freigeist-Fellowship of the Volkswagen Stiftung, and the German research council (DFG) Research Training Group 2236 UnRAVeL.
Simulations were performed with computing resources granted by RWTH Aachen University.

\bibliography{Bibliography/references.bib}

\end{document}